\documentclass[11pt]{article}
\usepackage{amsmath,amssymb,amsthm}
\usepackage{enumerate}
\usepackage{geometry}
\numberwithin{equation}{section}
\newtheorem{theorem}[equation]{Theorem}
\newtheorem{cor}[equation]{Corollary}
\newtheorem{lemma}[equation]{Lemma}
\theoremstyle{definition}
\newtheorem{remark}[equation]{Remark}

\newcommand{\la}{\langle}
\newcommand{\ra}{\rangle}

\newcommand{\R}{\mathbb{R}}
\newcommand{\calC}{\mathcal{C}}

\def\div{\operatorname{div}}
\DeclareMathOperator{\rot}{rot} 

\def \bX{{\mathbf X}}
\def \LL{L^2_0(\Omega)}
\def \bL{{\mathbf L}}
\def \bW{{\mathbf W}}

\def \bH{{\mathbf H}}
\def \bC{{\mathbf C}}
\def \bu{{\mathbf u}}
\def \ba{{\mathbf a}}
\def \blf{{\mathbf f}}

\newcommand{\bpsi}{{\boldsymbol \psi}}
\newcommand{\bchi}{{\boldsymbol \chi}}

\def \bw{{\mathbf w}}
\def \bn{{\mathbf n}}
\def \be{{\mathbf e}}
\def \bl{{\mathbf l}}
\def \dO{{\partial\Omega}}

\newcommand{\bv}{{\bf v}}
\newcommand{\bg}{{\bf g}}

\newcommand{\diff}{\, \mbox{\rm d}}

\title{On well-posedness of a velocity-vorticity formulation of the Navier--Stokes equations with no-slip boundary conditions}

\author{Maxim A. Olshanskii\thanks{Department of Mathematics, University of Houston, Houston, TX, 77204;
email: molshan@mat.uh.edu, partially supported by NSF Grant DMS 1522252 and ARO 65294-MA.}
\and
Leo G. Rebholz\thanks{Department of Mathematical Sciences, Clemson University, Clemson, SC, 29634;
email: rebholz@clemson.edu, partially supported by NSF Grant DMS 1522191 and U.S. Army Grant 65294-MA.}
\and
Abner J. Salgado\thanks{Department of Mathematics, University of Tennessee, Knoxville, TN 37996, USA.
email: asalgad1@utk.edu, partially supported by NSF Grant DMS 1418784.}
}

\begin{document}
\date{}
\maketitle

\begin{abstract}
We study well-posedness of a velocity-vorticity formulation of the Navier--Stokes equations, supplemented with no-slip velocity boundary conditions, a no-penetration vorticity boundary condition, along with a natural vorticity boundary condition depending on a pressure functional. In the stationary case we prove existence and uniqueness  of a suitable weak  solution to the system under a small data condition. The topic of the paper is driven by recent developments of vorticity based numerical methods for the Navier--Stokes equations.
\end{abstract}

\section{Introduction}
\label{sec:Intro}

The evolution of an incompressible, viscous Newtonian fluid is governed by the
Navier--Stokes equations (NSE), which for a given bounded, connected domain $\Omega\subset\mathbb{R}^3$ with a piecewise, smooth, Lipschitz boundary $\partial\Omega$, an end time $T$, an initial condition $\bu_0:\Omega \to \mathbb{R}^3$ and force field $\blf:(0,T]\times\Omega\rightarrow \mathbb{R}^3$ read: find a velocity field $\bu:[0,T]\times\Omega\rightarrow \mathbb{R}^3$ and a pressure field $p:(0,T]\times\Omega\rightarrow \mathbb{R}$ such that
\begin{equation}
\label{eq:NSE}
\begin{aligned}
  \frac{\partial \bu}{\partial t} + (\bu \cdot\nabla)\bu - \nu \Delta \bu  + \nabla p = \blf &\quad \text{in } (0,T]\times\Omega,\\
  {\nabla\cdot\bu} = 0 &\quad \text{in } (0,T]\times\Omega,\\
  \bu|_{t=0}=\bu_0 &\quad \text{in } \Omega,\\
\end{aligned}
\end{equation}
where $\nu>0$ is the kinematic viscosity. The system must also be equipped with appropriate boundary conditions.

For suitably smooth solutions, equations \eqref{eq:NSE} can be re-formulated in other equivalent forms, in particular through the introduction of new variables such as vorticity, streamfunction,
rate of strain tensor, Bernoulli pressure, etc., see, e.g., \cite{GR86,G89,majda2002vorticity}.
While different formulations often provide useful insights into both physical and mathematical properties of 
NSE solutions,
it is believed that in general they do not include  more information than that already contained in \eqref{eq:NSE}. The situation changes, however, if one is interested in numerical solutions. It is well known that by applying a discretization method to different equivalent formulations of the equations, one commonly obtains \emph{non}-equivalent discrete problems with strikingly different numerical properties and, for particular problems, appropriate discretizations of certain formulations outperform others.
For this reason, the literature on computational methods for \eqref{eq:NSE} often considers
formulations different from the primitive variables ``convective'' formulation in \eqref{eq:NSE}.  Among these, we can mention the conservative form, skew-symmetric form, vorticity-streamfunction formulation, streamfunction formulation, rotational form, and EMAC form. We refer the reader to \cite{CHOR16,GS98,laytonbook,T77} for a description of these formulations and their discretizations.

For numerical analysis of discretization methods, such as a finite element method,  one  typically needs  discrete counterparts of fundamental  a priori bounds and well-posedness results. However due to the reasons outlined above, for discrete solutions such results often cannot be obtained by resorting to the primitive variables that describe the equations  in \eqref{eq:NSE}. This is the main motivation of this work where, through direct arguments, we show well-posedness of a particular formulation of NSE based on vorticity dynamics.
Our analysis overcomes two major (technical) difficulties:
\begin{enumerate}[(i)]
  \item The \emph{implicit} enforcement of the relation $\bw={\rot}\,\bu$ between two independent variables, vorticity and velocity of the fluid.  In fact, our arguments extend to the more general case when this identity is no longer true.

  \item Imposing the practically important no-slip conditions on solid walls on the velocity $\bu$ leads to natural boundary conditions for the vorticity $\bw$, which involve boundary values of the pressure variable.
\end{enumerate}
The two observations above have two immediate consequences: first, the energy estimate for $\bu$ does not necessarily  imply any bound for $\bw$; and, second, one has to look for more regular weak solutions to give sense to the vorticity boundary conditions on solid walls.
A key new technical result in this paper is the well-posedness and regularity of a Stokes-type problem with non-standard boundary conditions. The problem was introduced in \cite{girault1990curl} but, to the best of our knowledge, it has not been further explored in the literature.

Velocity-vorticity formulations  have been widely used in the numerical literature  and have been found to provide greater accuracy in simulations (compared to more common formulations), especially for flows where boundary effects are critical \cite{WB02,LYM06,MF00,G91,GHH90,WWW95,HOR17}.
Such formulations directly build on the equation of the  vorticity dynamics,
\begin{equation}\label{vort}
\frac{\partial {\bw}}{\partial t} - \nu \Delta {\bw} +
(\bu \cdot \nabla) {\bw}  - (\bw \cdot\nabla) \bu =
 {\rot}\,\blf.
\end{equation}
There are several ways to close the system of equations for $\bw$ and $\bu$. One common way is to  supplement \eqref{vort} with the vector Poisson problem for the velocity, $\Delta\bu={\rot}\,\bw$. In this paper, we consider the formulation that couples  \eqref{vort} with the rotational form of the momentum equation,
\begin{equation}\label{moment}
\frac{\partial \bu}{\partial t} - \nu \Delta \bu +\bw\times \bu + \nabla P =\blf.
\end{equation}
We are motivated to follow this path by the recent development of numerical methods for this coupling \cite{OR10,LOR11,BORW12,palha2017mass,HOR17}. As these references show, these schemes provide superior long-time numerical stability and more accurate fulfillment of discrete conservation laws (in addition to other well-known benefits of vorticity based numerical formulations).

The system of velocity-vorticity equations should be supplemented with suitable boundary conditions for both $\bu$ and $\bw$. The most practically important are no-slip, no-penetration conditions on solid boundaries of the fluid domain, i.e., $\bu={\bf 0}$. The corresponding vorticity boundary conditions are not straightforward to obtain and were actually the subject of a long and controversial discussion in the literature; see, e.g., \cite{GHOR15} for a brief overview. The normal condition $\bw\cdot\bn=0$ on $\dO$ is the immediate consequence of $\bu={\bf 0}$ and $\bw=\rot\bu$ for $\bu$ smooth up to the boundary. Here and further $\bn$ is the unit normal vector on $\dO$.   To define tangential vorticity boundary conditions, we follow \cite{GHOR15}, where the following physically consistent boundary condition on solid walls was derived,
\begin{equation}\label{bc}
\nu (\rot \bw)\times\bn =
({\blf}-\nabla P)\times\bn.
\end{equation}
This boundary condition can be obtained by taking the tangential component of \eqref{moment} on $\dO$, and recalling that $\Delta\bu=\rot\bw$ and $\bu={\bf 0}$ on $\dO$. It can be shown, cf. \cite{GHOR15}, that for $\blf={\bf 0}$, condition \eqref{bc} characterizes  the production of the {streamwise} and {spanwise}  vorticity on  no-slip boundaries due to the tangential pressure gradient and vorticity  intensification, depending on the shape of $\dO$. The vorticity boundary condition \eqref{bc} was used in \cite{GHOR15,HOR17} to produce numerical simulations with excellent results.
We shall see that the condition in \eqref{bc} is a natural boundary condition, which enters the weak formulation as a functional.

Based on the numerical success of the velocity-vorticity  formulation \eqref{vort}--\eqref{moment} with boundary conditions $\bu={\bf 0}$, $\bw\cdot\bn=0$ on $\dO$, and \eqref{bc}, a theoretical study of the underlying PDE system is warranted.  Hence, the purpose of this work is to initiate the study of the velocity-vorticity system equipped with no-slip velocity boundary conditions, and the corresponding vorticity boundary conditions.  We shall focus on a stationary version of the problem:
\begin{equation}
\begin{split}
 \alpha\bu- \nu \Delta {\bu} + \bw\times {\bu} + \nabla P &= \blf, \\
 \div \bw= \div \bu &= 0,\\
 \alpha\bw- \nu \Delta \bw +
(\bu \cdot \nabla) \bw  - (\bw \cdot\nabla) \bu +\nabla \eta &=
 \rot \blf, \\
 \bu |_{\partial\Omega}& ={\bf 0},\\
 \bw\cdot\bn  |_{\partial\Omega} &=0, \\
 \left(\nu \rot \bw\times\bn -
(\blf-\nabla P)\times\bn \right)  |_{\partial\Omega} &= {\bf 0},
\end{split}
\label{VVS}
\end{equation}
with $\alpha\ge0$.
For this system we will show existence and uniqueness of solutions, provided that the problem data is sufficiently small, or $\alpha$ is sufficiently large.  We introduced the non-negative parameter $\alpha$ in the system, with the aim of covering the case when an additional zero order term results from the approximation of time derivative by a finite difference. One, of course, would be also interested in treating  the true semi-discrete problem and passing to the limit of time step turning to zero. However, this seems beyond our reach, and so we cannot provide an existence result for the full time dependent problem. This, in particular, is due to the presence of the pressure in the vorticity boundary conditions \eqref{bc}.
Nevertheless, the partial analysis that we provide here comes with its own set of unique features. As mentioned above, to overcome them, we need to analyze a class of Stokes-type problems with nonstandard boundary conditions, and Helmholtz-Weyl decompositions of the space $\bL^r(\Omega)$, $r \neq 2$.

\begin{remark}[{solenoidal vorticity}]
We note that \eqref{VVS} explicitly enforces the divergence free condition for vorticity  with the help of the Lagrange multiplier $\eta$ in the vorticity equation. There are several reasons for doing this. First, we find that this helps us in the analysis. Second, this is motivated by numerical experience, since a common finite element method for \eqref{vort} does not yield divergence free discrete solutions and enforcing it as a constraint benefits the accuracy of the recovered vorticity field. Finally, all the arguments in this paper extend now in a straightforward way to a more generic system of equations, where $\rot \blf$ on the right-hand side of the vorticity equation in \eqref{VVS} is replaced by a sufficiently regular source term $\mathbf{f_w}$, not necessary divergence free, which can be interpreted as modeling external sources of vorticity production.

We also note the alternative way of enforcing $\div\bw= 0$ by re-writing the vorticity equation \eqref{vort} in the form
\begin{equation}
\frac{\partial {\bw}}{\partial t} - \nu \Delta {\bw} + 2D(\bw)\bu - \nabla \eta = \rot \blf.
\label{VVH2}
\end{equation}
In this equation, $\eta$ has the physical meaning of the \emph{helical density} (rather than an auxiliary variable) and $D(\bw)\bu = \bu \cdot \nabla \bw - \tfrac12 (\rot\ \bw) \times \bu$.  When \eqref{VVH2} is used as the vorticity equation in \eqref{VVS}, the system is called the velocity-vorticity-helicity (VVH) formulation of the NSE~\cite{OR10}.  It is unique in the fact that it naturally enforces the {mathematical constraint} that the vorticity be divergence-free, and also that it solves for helicity directly (via helical density). The analysis of the paper can be extended to the system \eqref{VVS}, where the vorticity equation is written in the vorticity--helical density form \eqref{VVH2}, but we shall not elaborate further on this.
\end{remark}

Our presentation is {organized} as follows. In section \ref{sec:prelim} we provide some preliminary results along with some new results concerning the Stokes problem with nonstandard boundary conditions. In section \ref{sec:stationary} we show that system \eqref{VVS} is well posed. The analysis follows carefully the dependence of all estimates on problem  parameters, since such dependence is of interest for understanding the properties of numerical solutions. In particular, we are interested in the dependence on $\alpha$, since it comes from the time-discretization and so is  ``user-controlled''. 
Conclusions and potential directions for future work are provided in section \ref{sec:conclusion}.

\section{Preliminaries}
\label{sec:prelim}
In this work we assume that the fluid occupies a bounded, simply connected domain $\Omega \subset \R^3$, with $\dO\in C^3$. We will follow standard notation concerning function spaces. In particular, we will denote the $L^2(\Omega)$ inner product and norm by $(\cdot,\cdot)$ and $\| \cdot \|$, respectively.  All other norms will be clearly labeled.

The natural function spaces for velocity, vorticity and pressure in the presence of solid walls are
\[
  \bX = \bH^1_0(\Omega), \quad
  \bW  = \left\{ \bv \in \bH^1(\Omega):\ \bv \cdot \bn|_{\partial\Omega} = 0 \right\}, \quad
  Q  =\LL\cap H^1(\Omega),
\]
where $\LL$ denotes the subspace of $L^2(\Omega)$ consisting of functions with zero mean. The Ne\v{c}as inequality implies that there exists $\beta_1>0$ such that
\begin{equation}
\beta_1 \| q \| \le \|\nabla q\|_{\bX'}\quad \forall~q\in\LL. \label{infsup1}
\end{equation}

\subsection{{The operators $\rot$ and $\div$}}
\label{sub:divcurl}

{Let us here recall some well-known results about the operators $\rot$ and $\div$, and explore some of their consequences that shall be useful in the sequel}. The assumptions on $\Omega$ yield that $\bW= \bH_0(\div)\cap\bH(\rot)$ both algebraically and topologically and, as shown in \cite[section 3.5]{GR86}, we additionally have that
\begin{align}
\|\bv\|&\le C_{0}(\|\rot\bv\|+\|\div\bv\|)~~\forall~\bv\in\bW, \label{DivRot}\\
\| \nabla \bv \|^2 &\le  C_1(  \| \rot \bv \|^2 +  \| \div \bv \|^2) ~~\forall~\bv\in\bW,\label{equiv}\\
\|\bv\|_{{ \bH}^2}&\le C_{2}(\|\bv\|+\|\rot\bv\|_{{ \bH}^{1}}+\|\div\bv\|_{{ \bH}^{1}})~~\forall~\bv\in\bW\cap {  \bH^2(\Omega)}. \label{DivRot2}
\end{align}
If $\Omega$ is convex, then $C_1=1$ in \eqref{equiv}.

Since $\Omega$ is assumed to be bounded, simply connected, and have a smooth boundary, we have the following result about the solvability of the equation $\rot\bv=\bu$, see \cite[Theorems 3.5 and 3.6]{GR86}.

\begin{lemma}[right inverse of $\rot$]
\label{LGirault}
Let $\bw\in \bH_0(\div)$ be divergence free, i.e. $\div\bw=0$ a.e. in $\Omega$. Then there is a unique vector potential $\bu\in \bH^1(\Omega)$ characterized by
\[
  \rot\bu=\bw,\quad \div\bu=0~~\text{in}~\Omega,\quad\bu\times\bn={\bf 0}~~\text{on}~\dO,\quad
 \|\bu\|_{\bH^1}\le c\|\bw\|.
\]
\end{lemma}

Notice that if $\bv\in\bH(\rot)$, then $\rot\bv\in\bH(\div)$, and so the condition $\bn\cdot\rot\bv=0$ on $\dO$ is well-defined. We therefore define the space
\[
\widetilde{\bH}(\rot):=\{\bv\in\bH(\rot)\,:\,\bn\cdot\rot\bv=0~~\text{on}~\dO\}.
\]
Owing to Lemma~\ref{LGirault}, we can characterize this space.

\begin{lemma}[characterization of $\widetilde{\bH}(\rot)$]
\label{L_decomp}
A function $\bv\in\widetilde{\bH}(\rot)$ iff the following decomposition holds
\begin{equation}\label{decomp}
\bv=\bu+\nabla q,\quad q\in H^1(\Omega),~\bu\in \bH^1(\Omega),~~s.t.~~ \div\bu=0~~\text{in}~\Omega,\quad\bu\times\bn=0~~\text{on}~\dO.
\end{equation}
\end{lemma}
\begin{proof}
The result of the Lemma is given in \cite[Remark 2.1]{girault1990curl}, but we prove it here for completeness.  For a given $\bv\in\widetilde{\bH}(\rot)$, let $\bw=\rot\bv$ in Lemma~\ref{LGirault} and note that $\rot(\bu-\bv)={\bf 0}$ implies that $\bu-\bv=\nabla q$ for some $q\in H^1(\Omega)$.  The reverse
implication, i.e. $\bu+\nabla q\in \widetilde{\bH}(\rot)$ for all $q\in H^1(\Omega)$, $\bu\in \bH^1(\Omega)$ satisfying the
conditions in \eqref{decomp},   is also straightforward since $\bu\times\bn={\bf 0}~~\text{on}~\dO$ implies $\bn\cdot\rot\bu=0$ on $\dO$.
\end{proof}

We also define the subspace of $\widetilde{\bH}(\rot)$ consisting of divergence free functions with zero normal component on $\dO$ as
\[
\widetilde{\bH}^0(\rot):=\{\bv\in\widetilde{\bH}(\rot)\,:\,(\bv,\nabla q)=0~~\forall~q\in H^1(\Omega)\}.
\]
Notice that if $\bv \in \widetilde{\bH}^0(\rot)$, we have in particular that $(\bv,\nabla q)=0~\forall~q\in \dot{C}(\Omega)$, where $\dot{C}(\Omega)$ denotes the set of $C^{\infty}(\Omega)$ functions with compact support. Hence $\div\bv=0$ in $\dot{C}(\Omega)'$ and so, owing to a density argument, $\div\bv=0$ in $L^2(\Omega)$.
Therefore, $\bv\in\widetilde{\bH}^0(\rot)$ implies $\bv\in \bH(\div)$ and the condition $\bn\cdot\bv=0$ holds in $H^{-\frac12}(\dO)$, and so we conclude $\bn\cdot\bv=0$ a.e. on $\dO$. Therefore, for $\bv\in\widetilde{\bH}^0(\rot)$, we get due to the regularity of $\dO$ that $\bv\in \bH^1(\Omega)$, and because of the boundary condition we have $\bv\in\bW$.
Summarizing, we have shown that
\begin{equation}\label{emb}
\widetilde{\bH}^0(\rot)\subset\{\bv\in\bW\,:\,\div\bv=0\}.
\end{equation}

Owing to the fact that we are assuming $\dO \in C^3$ we have the following Helmholtz-Weyl decomposition of $\bu\in\bL^r(\Omega)$, $r\in(1,\infty)$, see \cite[Theorems III.1.2 and III.2.3]{galdi2011introduction}:
\begin{equation}
\label{HW}
  \bu=\nabla q+\bpsi,
\end{equation}
where $q \in W^{1,r}(\Omega)/\R$ is unique,
\[
  \bpsi\in \bH_r^0(\div):=\{\bpsi\in\bL^r(\Omega) \,:\,\div\bpsi=0,~\bpsi\cdot\bn|_{\dO}=0\},
\]
and $\|\bpsi\|_{\bL^r}+\|\nabla q\|_{\bL^r}\le c_0\|\bu\|_{\bL^r}$.
The existence of the Helmholtz-Weyl decomposition \eqref{HW} implies the well-posedness of the following problem \cite{galdi2011introduction}:
Given $\bg\in\bL^r(\Omega)$ with $r \in (1,\infty)$, find $\psi\in W^{1,r}(\Omega)/\mathbb{R}$ such that
\begin{equation}
\label{P_Pois}
(\nabla \psi,\nabla \phi)=(\bg,\nabla \phi)\quad \forall~\phi\in  W^{1,t}(\Omega),\quad \frac1r + \frac1t = 1. 
\end{equation}
The space $\bH_r^0(\div)$ from the Helmholtz-Weyl decomposition admits the following characterization for simply connected domains $\Omega$ for which $\dO\in C^2$, see \cite[Corollary 2.3]{borchers1990equations}.

\begin{lemma}[characterization of $\bH_r^0(\div)$]
\label{lem:HW1}
Let $\Omega$ be bounded, simply connected with $\dO \in C^2$ and connected. If $\bu\in\bH_r^0(\div)$ with $r \in (1,\infty)$, then there is a unique $\bv \in \bW^{1,r}_0(\Omega)$ such that
\begin{equation}
\label{HW1}
  \bu = \rot \bv, \quad
  \Delta\div \bv = 0 ~ \text{in} ~\Omega,
\end{equation}
with the estimate $ \| \bv \|_{\bW^{1,r}} \leq c \| \bu \|_{\bL^r}$.
\end{lemma}

This characterization yields the following result.

\begin{cor}[gradient estimate]
\label{LRotEst}
In the setting of Lemma \ref{lem:HW1}, if $\bu\in\bH_r^0(\div)$ is such that $\rot\bu\in\bH_r^0(\div)$, then $\bu\in \bW^{1,r}(\Omega)$ and
$\|\nabla \bu\|_{\bL^r}\le c\,\|\rot\bu\|_{\bL^r}$.
\end{cor}
\begin{proof} Note that $\rot\bu\in\bH_r^0(\div)$ implies $\rot\bu=\rot\bv$ with $\bv$ as in \eqref{HW1}, satisfying
$\|\nabla \bv\|_{\bL^r}\le c\,\|\rot\bu\|_{\bL^r}$.

Define $\be=\bu-\bv$, and notice that we have $\div\be=\div\bv$, $\be\cdot\bn=0$ on $\dO$, and $\rot\be={\bf 0}$. This, in particular, implies that $\be=\nabla\psi$. Therefore, the function $\psi$ solves the Neumann problem $\Delta\psi=-\div\bv$ in $\Omega$ with $\bn\cdot\nabla\psi=0$ on $\dO$. Using a regularity result for this problem, see \cite[Theorem 2.4.2.7]{grisvard2011elliptic}, we get
\[
\|\nabla \bu\|_{\bL^r}\le \|\nabla \bv\|_{\bL^r}+\|\nabla^2\psi\|_{\bL^r}\le c\,(\|\rot\bu\|_{\bL^r}+\|\div\bv\|_{\bL^r})
\le c\,\|\rot\bu\|_{\bL^r},
\]
where we used that $\|\nabla \bv\|_{\bL^r}\le c\,\|\rot\bu\|_{\bL^r}$.
\end{proof}

\subsection{The Stokes problem with nonstandard boundary conditions}
\label{sub:Stokes}

Consider first the classical Stokes problem, supplemented with no-slip boundary conditions:
\begin{equation}
\label{byhand3}
  \begin{aligned}
  -\nu \Delta \bu + \nabla p&={\bf F} \quad \text{in }\Omega,\\
  \div \bu &= 0 \quad \text{in } \Omega,\\
  \bu &= {\bf 0} \quad \text{on } \partial \Omega.\\
  \end{aligned}
\end{equation}
Recall the following regularity result, \cite[Lemma IV.6.1]{galdi2011introduction}, if $\partial \Omega \in C^2$ and ${\bf F} \in \bL^{r}{(\Omega)}$ with $r\in (1, \infty)$, then the solution to \eqref{byhand3} satisfies $\bu \in \bW^{2, r}(\Omega)$, $p \in W^{1, r}(\Omega)/\R$, and
\begin{equation}
\label{eq:CZStokes}
  \nu \| \bu \|_{\bW^{2, r}} + \|  {p} \|_{W^{1, r}/\R}  \leq c \| {\bf F} \|_{\bL^r}.
\end{equation}

Our analysis in section \ref{sec:stationary} will also require regularity results for the following non-standard Stokes-type problem,
\begin{equation}\label{Stokes_non}
\begin{aligned}
  \alpha\bu-\nu \Delta\bu + \ba\times\rot\bu +\nabla p&= \bg \quad \text{in } \Omega,\\
  \div\bu&=0 \quad \text{in } \Omega,\\
  \bn\cdot\rot\rot\bu=\bn\cdot\rot\bu=\bu\cdot\bn &= {0}\quad \text{on } \partial\Omega,
\end{aligned}
\end{equation}
with $\ba\in\bX$ and $\div\ba=0$.
The problem with $\ba={\bf 0}$, $\alpha=0$ was discussed in~\cite{girault1990curl}.

We will now establish the necessary well-posedness and regularity for \eqref{Stokes_non}. A weak formulation reads: For $\bg\in \bL^2(\Omega)$ find $(\bu,p)\in\widetilde{\bH}(\rot)\cap\bL^3(\Omega)\times W^{1,\frac32}(\Omega)/\mathbb{R}$ satisfying
\begin{equation}\label{weak_non}
  \begin{aligned}
    \alpha(\bu,\bpsi)+ \nu (\rot\bu,\rot\bpsi) +(\ba\times\rot\bu,\bpsi) +(\nabla p,\bpsi)&= (\bg,\bpsi), \\
    (\nabla q,\bu)&=0,
  \end{aligned}
\end{equation}
for all $(\bpsi,q)\in\widetilde{\bH}^{}(\rot)\cap\bL^3(\Omega)\times W^{1,\frac32}(\Omega)$.

In light of the decomposition given in Lemma \ref{L_decomp} the variables $\bu$ and $p$ can be decoupled and we can consider, instead, the following weak formulation: find $(\bu,p)\in\widetilde{\bH}^0(\rot)\times W^{1,\frac32}(\Omega)/\mathbb{R}$ that solve
\begin{equation}\label{weak_non2}
  \begin{aligned}
    \alpha(\bu,\bpsi)+ \nu (\rot\bu,\rot\bpsi) +(\ba\times\rot\bu,\bpsi)&= (\bg,\bpsi),\quad \forall~\bpsi\in\widetilde{\bH}^0(\rot),\\
    (\nabla p,\nabla q)&=(\bg-\ba\times\rot\bu-\alpha\bu,\nabla q),\quad \forall~q\in W^{1,3}(\Omega).
  \end{aligned}
\end{equation}
To see this it suffices to set, in \eqref{weak_non}, $\bpsi = \bv \in \widetilde{\bH}^0(\rot)$ and $\bpsi = \nabla q$ with $q \in W^{1,3}(\Omega)$ and recall that \eqref{emb} implies that $\widetilde{\bH}^0(\rot)\subset\bL^3(\Omega)$.


{We will now prove} the well-posedness of \eqref{weak_non} and \eqref{weak_non2}. {This is done as an auxiliary step towards showing that}
\eqref{VVS} is well-posed for $\alpha$ large enough (if the other data is fixed). {For this reason}, we need to make sure that certain constants in our estimates are independent  of $\alpha\ge0$. Since this includes
two extreme cases: $\alpha=0$ and $\alpha\to+\infty$,  it is helpful to introduce the following parameter,
\begin{equation}
\label{eq:defofalphaplus}
  \alpha_+=\max\{\alpha,C_P^{-2}\nu\},
\end{equation}
where $C_P$ is the optimal constant in the Poincar\'e inequality $\|\bv\|\le C_P\|\nabla\bv\|$, $\bv\in\bW$. We introduce it in the definition of $\alpha_+$ {so that it has} the proper (physical) scaling. The definition of $\alpha_+$ allows {us to obtain} the following simple bound:
\begin{equation}\label{alpha_b}
|(\bg,\bv)|\le \|\bg\|\min\{\|\bv\|,C_P\|\nabla\bv\|\}\le \sqrt{2}\alpha_+^{-\frac12}\|\bg\|\|\bv\|_\ast\quad\forall\bg\in\bL^2,~\bv\in\bW,
\end{equation}
with
\begin{equation}
\label{eq:defofastnorm}
  \|\bv\|_\ast:=(\alpha\|\bv\|^2+\nu\|\nabla\bv\|^2)^{\frac12}.
\end{equation}
To check the last inequality in \eqref{alpha_b} it is helpful to note the trivial bound $\min\{a,b\}\max\{c,d\}\le
ac+bd\le\sqrt{2}((ac)^2+(bd)^2)^{\frac12}$, for non-negative reals $a,b,c,d$. {Notice, in addition, that by} considering two cases $\alpha\le C_P^{-2}\nu$ and $\alpha> C_P^{-2}\nu$ and also using Poincare's or Young's inequality, {we obtain that} 
\begin{equation}\label{alpha_b2}
\|\bv\|\|\nabla\bv\|\le  \alpha_+^{-\frac12}\nu^{-\frac12}\|\bv\|_\ast^2\quad\forall~\bv\in\bW.
\end{equation}
These two inequalities will {help} us to obtain conditions to guarantee that problems \eqref{weak_non} and \eqref{weak_non2} are well posed.

\begin{lemma}[{conditional} well-posedness]
\label{Lreg}
There is a constant $C^*$ which depends only on $\Omega$ such that, if
\begin{equation}
  \label{eq:exisence_condition}
  \|\nabla\ba\|\leq C^* \nu^{3/4} \alpha_+^{1/4}
\end{equation}
then problems \eqref{weak_non} and \eqref{weak_non2} are well-posed and  have the same unique solution.
\end{lemma}
\begin{proof}
We begin the proof by noting that all bilinear forms in \eqref{weak_non} and \eqref{weak_non2} are well-defined and continuous on the corresponding spaces. In particular, we have
\begin{equation}
\label{eq:rhsbdd}
|(\ba\times\rot\bu,\bpsi)|\le \|\ba\|_{\bL^6}\|\rot\bu\|\|\bpsi\|_{\bL^3}\le C_3\|\nabla\ba\| \|\rot\bu\|\|\bpsi\|_{\bL^3}.
\end{equation}

The Gagliardo--Nirenberg interpolation inequality \cite[Theorem 4.17]{Ada75} provides
\begin{equation}
\label{aux33}
  \| \bu \|_{\bL^3} \leq \tilde C_5 \| \bu\|^{1/2} \| \nabla \bu \|^{1/2} \leq C_5 \| \bu\|^{1/2} \| \rot \bu \|^{1/2},\quad {\bu \in \widetilde{\bH}^0(\rot)},
\end{equation}
where in the last step we applied \eqref{equiv}. 
{Now \eqref{eq:rhsbdd}}--\eqref{aux33} yields the bound for $\bu \in \widetilde{\bH}^0(\rot)$,
\begin{equation}
  |(\ba\times\rot\bu,\bu)| 
  \leq C \| \nabla \ba \| \| \rot \bu \|^{3/2} \| \bu \|^{1/2}  \le \frac{C}{\alpha^{1/3}} \| \nabla \ba \|^{4/3} \| \rot \bu \|^2 + \frac\alpha2 \| \bu \|^2,
\label{aux34}
\end{equation}
where the constant $C$ depends only on $\Omega$. In view of {estimate \eqref{eq:rhsbdd} for $0 \leq \alpha < \alpha_+$ and \eqref{aux34} for $\alpha = \alpha_+$}, respectively, we see that there is a constant $C^*=C^*(\Omega)$ such that, whenever \eqref{eq:exisence_condition} holds, we have, for every $\bu\in\widetilde{\bH}^0(\rot)$, that
\begin{multline}\label{aux5}
\alpha(\bu,\bu)+\nu (\rot\bu,\rot\bu) +(\ba\times\rot\bu,\bu)\ge \frac\alpha2 \|\bu \|^2 + \frac\nu 2\|\rot\bu\|^2 \\
\ge  \frac\alpha2 \|\bu \|^2 + c\nu(\|\rot\bu\|^2+\|\bu\|_{\bL^3}^2),
\end{multline}
where the constant $c>0$ depends only on $\Omega$.
With these estimates at hand, we can now show the well posedness of each problem.

We begin with \eqref{weak_non2} since it is simpler. Notice that, in light of \eqref{aux5} and the Banach--Ne{\v c}as--Babu{\v s}ka theorem {\cite[Theorem 2.6]{ern2013theory}}, the $\bu$-problem is well posed. Now since $\bu \in \widetilde{\bH}^0(\rot)$ is uniquely defined, estimate \eqref{eq:rhsbdd} together with the well posedness of \eqref{P_Pois} show that the $p$-problem  is well-defined as well. In addition, {estimates \eqref{alpha_b} and \eqref{aux5}} yield
\begin{equation}\label{u_alpha}
c\|\bu \|^2_\ast\le |(\bg,\bu)|\le \sqrt{2}\alpha_+^{-\frac12}\|\bg\|\|\bu\|_\ast~~\Rightarrow~~\alpha\|\bu \| + \alpha_+^{\frac12}\nu^{\frac12}\|\nabla\bu \|\le C\, \|\bg\|.
\end{equation}
From the second equation in \eqref{weak_non2}, \eqref{eq:rhsbdd}   and  \eqref{u_alpha} we  conclude {that, if \eqref{eq:exisence_condition} holds, we have}
\begin{equation}\label{p_alpha}
  \begin{aligned}
    \|\nabla p \|_{L^{\frac32}} &\le \|\bg\|_{L^{\frac32}}+\alpha\|\bu\|_{L^{\frac32}}+ \|\ba\times\rot\bu\|_{L^{\frac32}} \\ &\le
    c(\|\bg\|+ \|\nabla\ba\| \|\nabla\bu\|)\le c(\|\bg\|+ \|\nabla\ba\|\nu^{-\frac12}\alpha_+^{-\frac12} \|\bg\|)\le c\|\bg\|,
  \end{aligned}
\end{equation}
{where, for the last inequality}, we used that $\nu^{-\frac12}\alpha_+^{-\frac12}\le c\,\nu^{-\frac34}\alpha_+^{-\frac14}$.

We now proceed with \eqref{weak_non}. For that we begin by noticing that both spaces $\widetilde{\bH}(\rot)\cap\bL^3(\Omega)$ and  $W^{1,\frac32}(\Omega)/\mathbb{R}$ are Banach and reflexive. Hence the well-posedness of \eqref{weak_non}  follows from the theory of saddle point problems as detailed, for instance, in \cite[Theorem 2.34]{ern2013theory}. Indeed, inequality \eqref{aux5} gives the inf-sup property and nondegeneracy of the $\bu$-form over the kernel of the $p$-form, which happens to coincide with $\widetilde{\bH}^0(\rot)\cap \bL^3(\Omega)$. On the other hand, the decomposition \eqref{decomp} and the Helmholtz-Weyl decomposition of $\bL^3(\Omega)$ yield the inf-sup property of the $p$-form:
\[
\begin{split}
  \sup_{\bpsi\in\widetilde{\bH}(\rot)\cap\bL^3(\Omega)}\frac{(\nabla p,\bpsi)}{\|\rot\bpsi\|+\|\bpsi\|_{\bL^3}}&\ge
  \sup_{q\in  W^{1,3}(\Omega)}\frac{(\nabla p,\nabla q)}{\|\nabla q\|_{\bL^3}}\ge
  c \sup_{\bv\in \bL^3(\Omega)}\frac{(\nabla p,\bv)}{\|\bv\|_{\bL^3}}=
  c\|\nabla p\|_{\bL^{\frac32}}\\ & \ge c\| p\|_{W^{1,\frac32}}\quad \forall~p\in W^{1,\frac32}(\Omega)/\mathbb{R}.
\end{split}
\]

It remains to show that the solutions coincide, but this is immediate upon choosing appropriate test functions.
\end{proof}

\begin{remark}[large $\alpha$]
Notice that, for any given viscosity $\nu$ and vector $\ba \in \bX$ with $\div \ba = 0$, there exists $\alpha$ large enough so that condition \eqref{eq:exisence_condition} is satisfied.  \end{remark}

We  now establish a \emph{regularity} result for the weak solution of \eqref{weak_non} and \eqref{weak_non2}. Although relatively straightforward, such a result is not found in  \cite{girault1990curl} nor seemingly anywhere else in the literature.

\begin{lemma}[regularity]
\label{Lreg1}
Assume {that} $\Omega$ is simply connected and $\dO\in C^3$. {If \eqref{eq:exisence_condition} holds, then} the solution to \eqref{weak_non} satisfies $(\bu, p)\in \bH^2(\Omega)\times H^1(\Omega)/\R$ and ${\nu} \|\bu\|_{\bH^2}+\|\nabla p\|\le C\,\|\bg\|$, {with $C=C(\Omega)$ independent of $\alpha$ and $\nu$}.
\end{lemma}
\begin{proof}
By definition, the solution to \eqref{weak_non2} satisfies $(\bu,p) \in\widetilde{\bH}^0(\rot)\times W^{1,\frac32}(\Omega)$. Thanks to the embedding \eqref{emb}, we have that $\div\bu=0$ in $\Omega$ and $\bu\cdot\bn=0$ on $\dO$.  We want to apply the regularity result in \eqref{DivRot2} and so it remains to show $\rot\bu\in { \bH^1(\Omega)}$ together with a suitable bound on $\|\rot\bu \|_{\bH^1}$.

First, we note that $\rot\bu\in\bL^2(\Omega)$ and $\ba\in\bL^6(\Omega)$ imply $ \bg-\nabla p-\ba\times\rot\bu-\alpha\bu \in\bL^{\frac32}{(\Omega)}$.
Thus, from {\eqref{weak_non}}, we have that $\rot\rot\bu\in \bL^{\frac32}(\Omega)$ and  $\bn\cdot\rot\rot\bu=0$ on $\dO$. Indeed, the integrability follows by taking $\bpsi\in \dot{\bC}(\Omega)$ in the first equation of \eqref{weak_non} and noting that
\[
  \nu \la\rot\rot\bu,\bpsi\ra_{\dot{\bC}'\times\dot{\bC}}= \nu (\rot\bu,\rot\bpsi)=(\tilde\bg,\bpsi)\le c\,\|\bpsi\|_{\bL^{3}},~~\text{with}~\tilde\bg=\bg-\nabla p-\ba\times\rot\bu -\alpha \bu.
\]
Since $\dot{\bC}(\Omega)$ is dense in $\bL^3(\Omega)$, $\rot\rot\bu$ defines a bounded linear functional on $\bL^3(\Omega)$, and hence $\rot\rot\bu\in (\bL^3(\Omega))'=\bL^{\frac32}(\Omega)$. By a similar argument, but now setting $\bpsi=\nabla\phi$ with $\phi\in C^\infty(\Omega)$ in \eqref{weak_non} we show $\bn\cdot\rot\rot\bu=0$ a.e. on $\dO$. Therefore, $\rot \bu$ satisfies, for $r=\frac32$, the assumptions of Corollary~\ref{LRotEst}, and we conclude $ \rot\bu\in \bW^{1,\frac32}(\Omega)$.

Having obtained this, let us focus now on obtaining a bound on $\| \tilde \bg\|_{\bL^{3/2}}$. In particular, H\"older's inequality and the embedding $H^1 \hookrightarrow L^6$ yield
\[
  \| \ba \times \rot \bu \|_{\bL^{3/2}} \leq \| \ba \|_{\bL^6} \| \rot \bu \| \leq C \| \nabla \ba \| \| \nabla \bu \|,
\]
for a constant $C$ that depends only on $\Omega$. Using now condition \eqref{eq:exisence_condition} and estimate \eqref{u_alpha} we obtain
\begin{equation}
\label{eq:arotu3over2}
  \|  \ba \times \rot \bu \|_{\bL^{3/2}} \leq  C \nu^{3/4} \alpha_+^{1/4} \| \nabla \bu \| \leq C \nu^{1/4} \alpha_+^{-1/4} \| \bg \|.
\end{equation}
{Using, again, estimate \eqref{u_alpha} in conjunction with \eqref{p_alpha} and \eqref{eq:arotu3over2} finally yields
\begin{equation}\label{aux493}
  \|\rot\bu\|_{\bL^{3}}\le c\|\rot\bu\|_{\bW^{1,\frac32}}\le c\, \nu^{-1}\|\tilde\bg\|_{\bL^{\frac32}} \leq C \nu^{-3/4} \alpha_+^{-1/4} \| \bg \|,
\end{equation}
for a constant $C$ that only depends on $\Omega$. We can now bootstrap this estimate to conclude that $\ba\times\rot\bu\in\bL^{2}$ with the estimate
\begin{equation}
\label{eq:auxAJS}
  \| \ba\times\rot\bu\| \leq \| \ba \|_{\bL^6} \| \rot \bu \|_{\bL^3} \leq C \nu^{3/4} \alpha_+^{1/4} \nu^{-3/4} \alpha_+^{-1/4} \| \bg \| \leq C \| \bg \|,
\end{equation}
where we also used \eqref{eq:exisence_condition}, and the constant depends only on $\Omega$.

The estimates above and the second equation in \eqref{weak_non2} imply the claimed regularity for the pressure: $\nabla p\in\bL^{2}(\Omega)$. Moreover, \eqref{u_alpha} and \eqref{eq:auxAJS} yield the estimate
\[
  \| \nabla p \| \leq \| \bg \| + \alpha \| \bu \| + \| \ba \times \rot \bu \| \leq C \| \bg \|,
\]
with a constant that depends only on $\Omega$.

It remains to show the regularity of $\bu$.}
Since $\nabla p, \bg, \ba\times\rot\bu, \alpha \bu\in \bL^2(\Omega)$ we have $\rot\rot\bu\in {  \bL^2(\Omega)}$.
Recalling that $\bn\cdot\rot\bu=0$ on $\dO$ we have $\rot\bu\in\bH(\rot)\cap\bH_0(\div)$ and so, owing to the regularity of $\dO$, $\rot\bu\in {  \bH^1(\Omega)}$, as we intended to show.
{Moreover, we have the estimate
\[
  \nu \|\bu\|_{\bH^2}\le c\nu \|\rot\rot\bu\|\le  c \left( \|\nabla p \|+ \|\bg\|+\alpha\|\bu\|+ \|\ba\times\rot\bu\| \right)\le C \|\bg\|,
\]
where $C$ only depends on $\Omega$.

This completes the proof.}
\end{proof}

{For given $\alpha \geq 0$ and $\ba \in \bX$, with $\ba$ solenoidal and satisfying \eqref{eq:exisence_condition}}
we denote by $\widehat{\Delta}^{-1}_{\ba}: \bL^2(\Omega) \rightarrow \bH^2(\Omega){  \cap}\bW$ the (velocity-)solution operator to the Stokes problem \eqref{Stokes_non}. {Owing to Lemma~\ref{Lreg1}, $\widehat{\Delta}_\ba^{-1}$ is well-defined}. {We will denote by} $\widehat{\Delta}^{-1}_{0}$ the solution operator with $\ba={\bf 0}$, $\alpha=0$. {This operator has} the following properties:
\begin{equation}\label{prop}
\begin{split}
{\nu \| \widehat{\Delta}^{-1}_{\ba} \bw \|_{{\bH^2}}  \leq C(\Omega) \| \bw \| },
\quad
\div \widehat{\Delta}^{-1}_{\ba}\bw=0,\quad
\bn\cdot\rot \widehat{\Delta}^{-1}_{\ba}\bw=0~\text{on}~\dO,\\
\alpha(\bw,\widehat{\Delta}^{-1}_{\ba}\bw)+ \nu(\rot\bw,\rot\widehat{\Delta}^{-1}_{\ba}\bw)+(\ba\cdot\nabla \bw-\bw\cdot\nabla \ba,\widehat{\Delta}^{-1}_{\ba}\bw)
=\|\bw\|^2~~\forall~\bw\in\bW,~s.t.~\div\bw=0,\\
(\bw,\widehat{\Delta}^{-1}_{\ba}\bw)= \nu \|\rot\widehat{\Delta}^{-1}_{\ba}\bw\|^2
  +\alpha\|\widehat{\Delta}^{-1}_{\ba}\bw\|^2
  +(\ba\times\rot\widehat{\Delta}^{-1}_{\ba}\bw,\widehat{\Delta}^{-1}_{\ba}\bw)~~\forall~\bw\in\bW,~s.t.~\div\bw=0.
 \end{split}
\end{equation}
The first three are obvious. For the fourth  one we note that $\bw\in\bW,~s.t.~\div\bw=0$ yields $\bw\in \widetilde{\bH}^0(\rot)$. Now the identity follows by taking $\bg=\bw$ and $\bpsi=\bw$ in \eqref{weak_non2}. 
The last one follows  by taking $\bg=\bw$ and $\bpsi=\widehat{\Delta}^{-1}_{\ba}\bw$ in \eqref{weak_non2}.

\subsection{Additional forms and their associated bounds}
\label{sub:forms}

As a last preliminary step, we define some forms that will be needed for the analysis of the {velocity-vorticity} formulation \eqref{VVS}. We first define the bilinear form {$f_{\rm bc}: Q\times \bW \to \mathbb{R}$} by
\[
f_{\rm bc}(P,\bchi):=\int_{\partial\Omega} (\nabla P\times \bn)\cdot \bchi \ \diff s.
\]
The form $f_{\rm bc}$ is well-defined on $Q\times \bW$ and continuous, since for any smooth $P$ and $\bchi$ it holds that
\begin{equation} \label{bdrybound}
\begin{split}
\int_{\partial\Omega} (\nabla P\times \bn)\cdot \bchi\ \diff s&=-\int_{\partial\Omega} (\nabla P\times \bchi)\cdot \bn\ \diff s=
-\int_{\Omega} \div(\nabla P\times \bchi)\, \diff s\\&
=\int_{\Omega} \nabla P \cdot  \rot \bchi  \diff x \le
\|\nabla P\|\| {\rot}\bchi\|.
\end{split}
\end{equation}
We also note the identities
\begin{equation*}
f_{\rm bc}(P,\bchi)
=\int_{\Omega} \nabla P \cdot  \rot \bchi \diff x = \int_{\partial\Omega} P \rot \bchi \cdot\bn\, \diff s,
\end{equation*}
which implies that, for any $\ba \in \bX$ with $\div \ba=0$ such that $\widehat{\Delta}_\ba^{-1}$ is well defined, we have
\begin{equation}\label{vanish}
f_{\rm bc}(P,\widehat{\Delta}^{-1}_{\ba} \bw)=0~~\forall~\bw\in\bW.
\end{equation}

Finally, define the trilinear form $b: {\bW}\times \bW \times \bW\rightarrow \mathbb{R}$ by
\[
  b(\bu,\bv,\bw):=(\bu\cdot\nabla \bv,\bw),
\]
and note that, whenever $\div \bu = 0$, this form is skew-symmetric, i.e. $b(\bu,\bv,\bv)=0$.

We will utilize the following bounds for the nonlinear terms that arise in our analysis.

\begin{lemma}[bounds on $b$]
\label{trilinbounds}
There exists a constant $M=M(\Omega)$  such that, for every $(\bu,\bv,\bw) \in \bW^3$ we have
\begin{align}
  | (\bw \times \bu,\bv)|& \leq M  \| \bw \| \| \nabla \bu \|^{\frac12}\| \bu \|^{\frac12} \| \nabla \bv \|, \label{trilin3b} \\
  |b(\bu,\bw,\widehat{\Delta}^{-1}_{0} \bv)|+|b(\bw,\bu,\widehat{\Delta}^{-1}_{0} \bv)|  & \leq M {\nu^{-1}} \| \nabla \bu \|^{\frac12}\| \bu \|^{\frac12} \| \bw \| \| \bv \|,~~\text{if}~\div\bu=\div\bw=0, \label{trilin5} \\
 |( \bw \times \bu, \rot \widehat{\Delta}^{-1}_{0} \bw)| & \le  M {\nu^{-1}} \| \bw \|^2  \| \nabla \bu \|^{\frac12}\| \bu \|^{\frac12}. \label{trilin6}
\end{align}
\end{lemma}
\begin{proof}
The constant $M$ will depend only on $\Omega$.  Each of the bounds in this lemma will hold with a potentially different constant that depends only on $\Omega$, and we take $M$ to be the maximum of these constants.
Estimate \eqref{trilin3b} follows from H\"older's inequality, Sobolev inequalities and \eqref{aux33}.

{For \eqref{trilin5}, to bound the first term, we use the fact that $\div \bu = 0$, H\"older's inequality and the embedding $H^1 \hookrightarrow L^6$ to obtain
\begin{align*}
  |b(\bu,\bw,\widehat{\Delta}^{-1}_{0} \bv)| &= \left|- (\bu \cdot \nabla \widehat{\Delta}^{-1}_0 \bv, \bw)  \right|
    \leq \| \bu \|_{\bL^3} \| \nabla \widehat{\Delta}^{-1}_0 \bv \|_{\bL^6} \| \bw \|
    \leq C\| \bu \|_{\bL^3} \| \widehat{\Delta}^{-1}_0 \bv \|_{\bH^2} \| \bw \| \\
  &\leq C\nu^{-1} \| \nabla \bu \|^{1/2} \|\bu\|^{1/2} \| \bv \| \| \bw \|,
\end{align*}
where in the last step we used \eqref{prop} and \eqref{aux33}}.
Similarly, for the second term, we obtain
\begin{align*}
  |b(\bw,\bu,\widehat{\Delta}^{-1}_{0} \bv)| & = \left|  (\bw\cdot\nabla \bu, \widehat{\Delta}^{-1}_{0} \bv) \right|
    = \left| (\bw\cdot\nabla \widehat{\Delta}^{-1}_{0} \bv,\bu ) \right|
    \leq C \| \bw \|\|\bu \|_{\bL^3}  \|\nabla \widehat{\Delta}^{-1}_{0} {\bv} \|_{\bL^6} \\
 & \leq C {\nu^{-1}} \| \nabla \bu \|^{\frac12}\| \bu \|^{\frac12} \| \bw \|\| \bv \|.
\end{align*}
Adding these bounds gives the result claimed in \eqref{trilin5}.

For \eqref{trilin6}, H\"older and Sobolev inequalities along with \eqref{aux33} and {\eqref{prop}} provide
\begin{align*}
 |( \bw \times \bu, \rot \widehat{\Delta}^{-1}_{0} \bw)| & \le \| \bw \| \| \bu \|_{\bL^3} \| \rot \widehat{\Delta}^{-1}_{0}\bw  \|_{{\bL^6}}
  \le C \| \bw \| \| \nabla \bu \|^{\frac12}\| \bu \|^{\frac12}  \| \widehat{\Delta}^{-1}_{0}\bw \|_{{\bH^2}} \\
 & \le C \nu^{-1} \| \bw \|^2 \| \nabla \bu \|^{\frac12}\| \bu \|^{\frac12},
 \end{align*}
which proves the stated result.
\end{proof}

\section{{Weak formulation of the velocity-vorticity problem} in the presence of solid walls and its well-posedness}
\label{sec:stationary}

In this section, we will prove that the velocity-vorticity formulation with no-slip velocity boundary conditions and the corresponding vorticity boundary conditions {\eqref{VVS}}
is well posed. {Our method of proof will be as follows: we will obtain a priori bounds which, through a fixed point argument, will allow us to show existence. In addition, provided the data satisfies further restrictions, we will be able to show uniqueness. The highlight of our approach lies in that our arguments hinge on tools developed in section \ref{sub:divcurl} and the problems studied in section \ref{sub:Stokes}.}

The weak formulation {of \eqref{VVS} is obtained in a standard way, that is multiplying by suitable test functions and integrating by parts. It reads:}
given $\blf\in \bL^2(\Omega)$,
find $(\bu,P,\bw, \eta)\in  {\bX\times Q \times \bW\times \LL}$ satisfying, for all $(\bv, \pi,\bchi,\lambda)\in  \bX\times Q \times \bW\times \LL$,
\begin{align}
\alpha(\bu,\bv)+ \nu(\nabla \bu,\nabla \bv) - (P,\div \bv) + (\bw \times \bu,\bv) &=   (\blf,\bv), \label{weak1} \\
(\div \bw, \lambda)=(\div \bu, {\pi}) & =  0,\\
\alpha(\bw,\bchi)+ \nu(\rot \bw,\rot \bchi) + \nu(\div \bw,\div \bchi)
+ b(\bu,\bw,\bchi)  \ \ \ \ \ & \\ - b(\bw,\bu,\bchi) - (\eta,\div \bchi)& =(\blf,\rot\bchi) - f_{\rm bc}(P, \bchi).\label{weak3}
\end{align}
We note that the term $\int_{\partial\Omega} (\blf \times \bn) \cdot \bchi \diff s$ coming from the vorticity boundary condition \eqref{bc} cancels, when we integrate by parts in $(\rot\blf,\bchi)$.  

\subsection{{A priori estimates and uniqueness}}
\label{sub:apriori_unique}

We begin the analysis of \eqref{weak1}--\eqref{weak3} with a priori bounds on solutions.

\begin{lemma}[a priori bounds]
\label{lem:aprioribounds}
Assume that system \eqref{weak1}--\eqref{weak3} has solutions. Then they satisfy
\begin{equation}
\label{eq:velestimate}
{ \| \bu \|_\ast \leq \sqrt2\alpha_+^{-1/2}  \| \blf \|},
\end{equation}
{where the norm $\|\cdot\|_\ast$ was defined in \eqref{eq:defofastnorm} and $\alpha_+$ in \eqref{eq:defofalphaplus}.}
If, in addition, the forcing term $\blf$ satisfies
\begin{equation}
\label{eq:f_small}
  \| \blf\| \leq \frac{C^*}{\sqrt{2}} \nu^{5/4} {\alpha_+^{3/4}},
\end{equation}
where the constant $C^*$ is the same as in \eqref{eq:exisence_condition} of Lemma \ref{Lreg}, then they verify
\begin{equation}
\label{eq:H2estsmalldata}
  \| \bw \| \leq {K_1 := \nu^{-1} C(\Omega)} \| \blf\|_{\bH^{-1}}, \quad
  { \nu \|\bu \|_{\bH^2} + \|\nabla P\| \leq K_2, \quad \| \bw \|_\ast +\|\eta \|  \leq K_3},
\end{equation}
{with $K_2=K_2(\nu, \alpha_+, \blf, \Omega )$ and $K_3 = K_3(\nu, \alpha_+, \blf, \Omega )$. The constants $K_2$ and $K_3$, however, remain uniformly bounded as $\alpha \to \infty$, while all the other data of the problem remain fixed}.
\end{lemma}
\begin{proof}
We begin with the velocity {bound}. Setting $\bv=\bu$ in \eqref{weak1} {and using \eqref{alpha_b}} yields estimate \eqref{eq:velestimate}.

Notice now that, if \eqref{eq:f_small} holds, then \eqref{eq:velestimate} implies that $\ba = \bu$ satisfies condition \eqref{eq:exisence_condition} of Lemma \ref{Lreg} and, consequently, the operator $\widehat{\Delta}^{-1}_{\bu}$ is well defined. We set $\bchi=\widehat{\Delta}^{-1}_{\bu} \bw$ in \eqref{weak3} and observe that, due to \eqref{vanish} and \eqref{prop}, the $\eta$-term and boundary functional vanish, thus \eqref{weak3} reduces to
\[
  \alpha (\bw, \widehat{\Delta}^{-1}_{\bu} \bw ) + \nu (\rot \bw, \rot \widehat{\Delta}^{-1}_{\bu} \bw) + b(\bu,\bw,\widehat{\Delta}^{-1}_{\bu} \bw) -
  b(\bw,\bu,\widehat{\Delta}^{-1}_{\bu} \bw) = (\blf,\rot \widehat{\Delta}^{-1}_{\bu} \bw),
\]
which when compared with \eqref{prop} yields
\[
  \| \bw \|^2 =  ( \blf, \rot \widehat{\Delta}^{-1}_{\bu} \bw)  \le \| \blf \|_{\bH^{-1}} \|\rot \widehat{\Delta}^{-1}_{\bu} \bw\|_{\bH^1}  \le C\| \blf \|_{\bH^{-1}} \|\widehat{\Delta}^{-1}_{\bu} \bw\|_{\bH^2}\le {\nu^{-1} C(\Omega)}\| \blf \|_{\bH^{-1}} \|\bw\|,
\]
{which is the first estimate in \eqref{eq:H2estsmalldata}}.

To obtain the velocity--pressure part of estimate \eqref{eq:H2estsmalldata} we employ a bootstrapping argument.
{H\"older's inequality, the embedding $H^1\hookrightarrow L^6$, the bound \eqref{eq:velestimate} on the velocity, and the $L^2$-bound on the vorticity yield the existence of a constant $C(\Omega)$ that depends only on the domain $\Omega$, for which
\begin{align*}
  \| \bw \times \bu \|_{\bL^{3/2}} &\leq \| \bw \| \| \bu \|_{\bL^6} \leq C(\Omega) \nu^{-1/2} \alpha_+^{-1/2} K_1 \| \blf \| =: \calC_0(\nu, \alpha_+, \blf, \Omega),
\end{align*}
and, consequently $\bw \times \bu \in \bL^{\frac32}(\Omega)$. We now apply the regularity result for the Stokes problem \eqref{eq:CZStokes} using $\bf F = \blf - \bw \times \bu-\alpha\bu \in \bL^{\frac32}(\Omega)$ with the bound
\[
  \|{\bf F} \|_{\bL^{\frac32}}  \leq C(\|\blf\| + \alpha\|\bu\| ) + \|\bw \times \bu\|_{\bL^{\frac32}}
  \leq C \| \blf \| + \calC_0(\nu, \alpha_+, \blf, \Omega) =: \calC_1(\nu, \alpha_+, \blf, \Omega),
\]
where we used \eqref{eq:velestimate} and the fact that $\alpha \leq \alpha_+$. The estimate given above shows that the velocity part of the solution to \eqref{weak1}--\eqref{weak3} satisfies
\begin{equation}\label{aux6}
\| \bu \|_{\bW^{2, \frac{3}{2}}}  \leq \nu^{-1} \calC_1(\nu, \alpha_+, \blf, \Omega ).
\end{equation}

We now invoke that, for every $r < \infty$, we have the embedding $W^{2,\frac32} \hookrightarrow L^r$ to obtain
\[
  \| \bw \times \bu \|_{\bL^{\frac74}} \leq \| \bw \| \| \bu \|_{\bL^{14}} \leq C \| \bw \| \| \bu \|_{\bW^{2,\frac32}}
    \leq C(\Omega) \nu^{-1} K_1 \calC_1(\nu, \alpha_+, \blf, \Omega ),
\]
where, in the last step, we used \eqref{aux6} and the $L^2$-estimate on the vorticity. This shows that ${\bf F} \in \bL^{\frac74}(\Omega)$ with the bound
\[
  \| {\bf F} \|_{\bL^{\frac74}} \leq C( \| \blf \| + \alpha \| \bu \| ) + \| \bw \times \bu \|_{\bL^{\frac74}}
  \leq C\| \blf \| + C(\Omega) \nu^{-1} K_1 \calC_1(\nu, \alpha_+, \blf, \Omega )
   =: \calC_2(\nu, \alpha_+, \blf, \Omega ).
\]
Using, once again, \eqref{eq:CZStokes} yields that $\bu \in \bW^{2,\frac74}(\Omega)$ with the estimate
\[
  \| \bu \|_{\bW^{2,\frac74}} \leq \nu^{-1} \calC_2(\nu, \alpha_+, \blf, \Omega ).
\]
Finally, we use the embedding $W^{2,\frac74} \hookrightarrow L^\infty$ to assert that
\[
  \| \bw \times \bu \| \leq \| \bw \| \| \bu \|_{\bL^\infty} \leq C(\Omega) \nu^{-1} K_1 \calC_2(\nu, \alpha_+, \blf, \Omega ).
\]
This gives us that ${\bf F} \in \bL^2(\Omega)$ with the estimate
\[
  \| {\bf F } \| \leq \| \blf \| + \alpha \| \bu \| + \| \bw \times \bu \| \leq C\| \blf \| + C \nu^{-1} K_1 \calC_2(\nu, \alpha_+, \blf, \Omega ),
\]
so that invoking, one last time, \eqref{eq:CZStokes} we obtain
\begin{equation}\label{aux2}
  \nu \| \bu \|_{\bH^2} + \| \nabla P\|  \leq C \| \blf \|  + C \nu^{-1} K_1 \calC_2(\nu, \alpha_+, \blf, \Omega )
  =: K_2(\nu, \alpha_+, \blf, \Omega ).
\end{equation}

Let us now bound the $H^1$-norm of the vorticity and the $L^2$-norm of $\eta$. Setting $\bchi = \bw$ in \eqref{weak3} yields
\[
  \alpha \| \bw \|^2 +\nu \| \rot \bw \|^2 + \nu \| \div \bw \|^2 = b(\bw,\bu,\bw) + (\blf,\rot\bw) - f_{\rm bc}(P,\bw).
\]
Applying \eqref{equiv} and then using the continuity of the functionals on the right hand side we obtain
\[
  \min\{1,C_1^{-1}\} \| \bw\|_\ast^2 \leq C( \|\nabla P \| + \| \blf \| ) \| \nabla\bw \| + b(\bw,\bu,\bw).
\]
To control the trilinear term, we use \eqref{eq:H2estsmalldata} and the embedding $H^2 \hookrightarrow W^{1,6}$ to obtain
\[
  |b(\bw,\bu,\bw)|  \leq \| \bw\| \| \nabla \bu \|_{\bL^6} \| \bw \|_{\bL^3} \leq C(\Omega) \nu^{-1} K_1 K_2 \| \nabla \bw \|,
\]
and, as a consequence,
\[
  \min\{ 1, C_1^{-1} \} \| \bw \|_\ast^2 \leq C \left( \| \blf \| + K_2(\nu, \alpha_+, \blf, \Omega ) + \nu^{-1} \| \blf \| K_2(\nu, \alpha_+, \blf, \Omega ) \right) \| \nabla \bw \|,
\]
from which the bound
\[
  \|  \bw \|_\ast \leq C \nu^{-1/2} \left( \| \blf \| + K_2+ \nu^{-1} K_1 K_2 \right) =: K_3(\nu, \alpha_+, \blf, \Omega )
\]
follows. An application of \eqref{infsup1} yields the desired bound for $\eta$.}

{It remains then to show that $K_2$ and $K_3$ are bounded as $\alpha$ grows large. To see this, we first observe that for $\alpha$ sufficiently large we have $\alpha_+ = \alpha$, and so we need to study the dependence on $\alpha_+$. In the course of the proof of estimates \eqref{eq:H2estsmalldata}, we obtained that
\[
  K_3 = C \nu^{-1/2}\left( \| \blf \| + K_2 + \nu^{-1} K_1 K_2 \right).
\]
Since $K_1$ does not depend on $\alpha_+$, it then is sufficient to show that $K_2$ remains bounded as $\alpha \to \infty$. From \eqref{aux2} we have
\begin{align*}
  K_2 &= C\left( \| \blf \| + \nu^{-1} K_1 \calC_2 \right) = C\left[ \| \blf \| + c\nu^{-1} K_1 \left( \| \blf \| + \nu^{-1} K_1 \calC_1 \right) \right] \\
  &= C\left\{ \| \blf \| + c\nu^{-1} K_1 \left[ \| \blf \| + \nu^{-1} K_1 \left( \| \blf \| + \nu^{-1/2} \alpha_+^{-1/2} \right) \right] \right\},
\end{align*}
where we successively applied the definitions of $\calC_j$, $j=0,1,2$. The only power of $\alpha_+$ that appears is negative and so we are able to conclude.}
\end{proof}

\begin{remark}[{large $\alpha$}]
If all the {problem} data besides $\alpha$ is {kept} fixed, the a priori bounds of Lemma \ref{lem:aprioribounds} will all hold if  $\alpha$ is taken sufficiently large.
\end{remark}

With the \emph{a priori} bounds on the solution of Lemma \ref{lem:aprioribounds}, we are able to prove uniqueness, under somewhat more restrictive conditions on the data.

\begin{theorem}[uniqueness]
\label{thm:uniqu_stat}
Assume that, {in addition to \eqref{eq:f_small}}, the problem data satisfies
\[
  { \alpha_1:=1-2 \sqrt{2} M\nu^{-\frac54}\alpha_+^{-\frac34} \| \blf \|>0 },
\]
and
\[
{2\sqrt{2} M^2 \nu^{-2}\alpha_+^{-1} \alpha_1^{-1} \| \blf \|  K_1 < 1}
\]
where {$K_1$}
is the data dependent constant from \eqref{eq:H2estsmalldata}
{for which, every vorticity solution to \eqref{weak1}--\eqref{weak3} verifies $\| \bw \|\le K_1$}.
Then solutions to \eqref{weak1}--\eqref{weak3} are unique.
\end{theorem}
\begin{proof}
Suppose there are two solutions to \eqref{weak1}--\eqref{weak3}, $(\bu_1,P_1,\bw_1,\eta_1),\ (\bu_2,P_2,\bw_2,\eta_2) \in \bX\times Q \times \bW\times\LL$, and set
\[
\be_u:=\bu_1 - \bu_2,\ e_P:=P_1 - P_2,\ \be_w:=\bw_1-\bw_2,\ e_\eta:=\eta_1 - \eta_2,.
\]
Subtracting \eqref{weak1}--\eqref{weak3} for each one of these two solutions gives, for all $(\bv,q,\bchi,\lambda)\in \bX\times Q \times \bW \times \LL$,
\begin{eqnarray*}
\alpha(\be_u,\bv)+\nu(\nabla \be_u,\nabla \bv) - (e_P,\div \bv) + (\bw_1 \times \be_u,\bv) +  (\be_w \times \bu_2,\bv) &=  &0, \\
(\div \be_w,\lambda)=(\div \be_u,q) & = & 0,\\
\alpha(\be_w,\bchi)+\nu(\rot \be_w,\rot \bchi) + \nu(\div \be_w,\div \bchi) + b(\be_u,\bw_2,\bchi) - b(\be_w,\bu_2,\bchi)  \\ + b(\bu_1,\be_w,\bchi) - b(\bw_1,\be_u,\bchi)- (e_\eta,\div \bchi)
& = & - f_{bc}(e_P,\bchi).
\end{eqnarray*}

Set now $\bchi=\widehat{\Delta}^{-1}_{0}\be_w$, which makes the pressure boundary term vanish, to obtain
\begin{equation*}
 \| \be_w \|^2 = - b(\be_u,\bw_2,\widehat{\Delta}^{-1}_{0}\be_w) + b(\be_w,\bu_2,\widehat{\Delta}^{-1}_{0}\be_w)  -  b(\bu_1,\be_w,\widehat{\Delta}^{-1}_{0}\be_w) + b(\bw_1,\be_u,\widehat{\Delta}^{-1}_{0}\be_w),
\end{equation*}
where we used \eqref{prop}. We now estimate each one of the trilinear terms on the right hand side of this identity
using \eqref{trilin5}, \eqref{alpha_b2} and \eqref{eq:velestimate} and obtain the bound
{
\begin{align*}
  \| \be_w \|^2 & \leq M \nu^{-1} \left( \| \nabla \be_u\|^{\frac12}\|\be_u\|^{\frac12} \| \bw_2 \| \| \be_w \|
    + \| \be_w \|^2 \| \nabla \bu_2\|^{\frac12}\| \bu_2\|^{\frac12} +  \| \nabla \bu_1 \|^{\frac12} \|\bu_1 \|^{\frac12}\| \be_w \|^2 \right. \\
  &+ \left. \| \nabla \be_u\|^{\frac12}\|\be_u\|^{\frac12}\| \bw_1 \| \| \be_w\| \right) \\
  &\leq M \nu^{-1} \left(  \| \nabla \be_u\|^{\frac12}\|\be_u\|^{\frac12} \| \bw_2 \|  + \| \nabla \be_u\|^{\frac12}\|\be_u\|^{\frac12}\| \bw_1 \| \right) \| \be_w \|
  + 2 \sqrt{2} M \nu^{-5/4} \alpha_+^{-3/4} \| \blf \| \| \be_w \|^2,
\end{align*}
which using the definition of $\alpha_1$, the bound given in \eqref{eq:H2estsmalldata}, and \eqref{alpha_b2} yields
\begin{equation}
\label{ewbound}
  \alpha_1 \| \be_w \| \leq M \nu^{-1} \left( \| \nabla \be_u\|^{\frac12}\|\be_u\|^{\frac12} \| \bw_2 \| + \| \nabla \be_u\|^{\frac12}\|\be_u\|^{\frac12}\| \bw_1 \| \right)
  \leq 2 M \nu^{-5/4} \alpha_+^{-1/4} K_1 \| \be_u \|_\ast.
\end{equation}}

Next, set $\bv=\be_u$ in the velocity error equation. This makes the pressure term  and one of the nonlinear terms vanish. Inequality \eqref{trilin3b}  then gives
\begin{equation}
\label{eubound}
  \begin{aligned}
    \|\be_u \|_\ast^2 &\le  M \| \be_w \| \| \nabla \bu_2 \| \| \nabla \be_u \|^{\frac12}\| \be_u \|^{\frac12}
 \le  M \nu^{-\frac14}\alpha_+^{-\frac14}\| \be_w \| \| \nabla \bu_2 \| \|\be_u \|_\ast  \\
 &\le  \sqrt{2} M \nu^{-\frac34}\alpha_+^{-\frac34} \| \blf \|  \| \be_w \| \|\be_u \|_\ast,
  \end{aligned}
\end{equation}
where, in the last step, we used the a priori bounds of Lemma~\ref{lem:aprioribounds} {and \eqref{alpha_b2}}.  Now using \eqref{ewbound} in \eqref{eubound}, we obtain
\[
 {\| \be_u \|^2_\ast  \le  2\sqrt{2} M^2 \nu^{-2}\alpha_+^{-1} \alpha_1^{-1} \| \blf \|  K_1  \| \be_u \|^2_\ast},
\]
which, by the second smallness assumption on the data yields {$\| \be_u \|_\ast=0$}.  From \eqref{ewbound} we immediately get that $\| \be_w \|=0$ and from Poincar\'e-Friedrichs' inequality $\|  \be_u \|=0$.   Now that we have
established $\be_u =0$ and $\be_w=0$, $e_P =0$ and $e_\eta=0$ follow from \eqref{infsup1}.
%
\end{proof}

\begin{remark}[{large $\alpha$}]
{As in Lemma \ref{lem:aprioribounds}}, if all the data besides $\alpha$ is fixed, Theorem \ref{thm:uniqu_stat} implies that uniqueness can be obtained by taking $\alpha$ sufficiently large.
\end{remark}

\subsection{Existence}
\label{sub:existence}

To prove the existence of solutions, we will utilize the following fixed point theorem, referred to as Shaefer's fixed point theorem in \cite{MR2597943} and as Leray-Schauder's fixed point theorem in \cite{GT83}.

\begin{lemma}[fixed point]
\label{LSlemma}
Let $Y$ be a real Banach space and $F:Y\rightarrow Y$ a compact map. Assume that the set of solutions to the family of fixed point problems:
\[
\mbox{find } y_{\lambda} \in Y \mbox{ satisfying } y_{\lambda} =\lambda F(y_{\lambda}), \ 0\le \lambda \le 1,
\]
are uniformly bounded. Then the problem $y^* = F(y^*)$ has a solution $y^*\in Y$.
\end{lemma}

We will proceed by constructing a compact map whose fixed points are solutions of \eqref{weak1}--\eqref{weak3}, then consider the family of fixed point problems, and finally apply Lemma \ref{LSlemma}.

Define $T:\bL^{2}(\Omega)\times\bW' \rightarrow \bX \times Q \times \bW\times\LL$ to be the solution operator of the following problem:
Find $(\bu,P,\bw,\eta)=T(\bg,\bl)\in  {\bX\times Q \times \bW\times \LL}$ satisfying for all $(\bv,{\pi},\bchi,\lambda)\in  {\bX\times Q \times \bW\times \LL}$
\begin{eqnarray}
\alpha (\bu, \bv) + \nu(\nabla \bu,\nabla \bv) - (P,\div \bv)  &=  & (\bg,\bv), \label{lin1} \\
(\div \bw, {\lambda})=(\div \bu, {\pi}) & = & 0, \label{lin2} \\
\alpha (\bw,\bchi)+ \nu(\rot \bw,\rot \bchi) + \nu(\div \bw,\div \bchi) - (\eta,\div \bchi)  & =&\bl(\bchi)- f_{\rm bc}(P,\bchi). \label{lin3}
\end{eqnarray}

\begin{lemma}[$T$ is well-defined]
\label{lem:gsmallTgood}
Let $(\bg,\bl) \in \bL^2(\Omega)\times\bW'$.  Then, problem \eqref{lin1}--\eqref{lin3} is well-posed and, as a consequence, $T$ is well-defined  and continuous. Moreover, $T(\bg,\bl)=(\bu,P,\bw,\eta)$ satisfies the following bounds:
\begin{eqnarray*}
  \| \nabla \bu \| & \le & \nu^{-1} \| \bg \|_{\bH^{-1}},\\
  {\nu}\| \bu \|_{\bH^2} + \| \nabla P \| & \le & C_S \| \bg \|,\\
 {\| \bw \|_\ast} +\|\eta\| & \le & C \left( \| \bl \|_{\bW'} + \| \bg \| \right),
\end{eqnarray*}
where $C$ and  $C_S$ are constants depending only on $\Omega$.
\end{lemma}
\begin{proof}
The bounds follow standard arguments. The first bound repeats the proof of \eqref{eq:velestimate}, while the second is \eqref{eq:CZStokes} for $r=2$. Once this is established, we invoke \eqref{equiv} to conclude the last one.
\end{proof}

Next, for $\blf \in \bL^2(\Omega)$  we define the nonlinear operator $N:\bX\times Q\times \bW\times\LL \rightarrow \bL^2(\Omega)\times\bW'$ as follows:
\begin{align*}
  N(\bu,P,\bw,\eta)_1 &:= \blf - \bw \times \bu  \in \bL^2(\Omega), \\
  \langle N(\bu,P,\bw,\eta)_2, \bchi \rangle_{\bW',\bW} &:=  (\blf,\rot\bchi)+b(\bw,\bu,\bchi)- b(\bu,\bw,\bchi)
\end{align*}
where we have that the first component $N(\bu,P,\bw,\eta)_1$ belongs to $\bL^2(\Omega)$ because we have that $\bu,\bw \in \bW \hookrightarrow \bL^4(\Omega)$. Moreover, the estimates of Lemma \ref{trilinbounds} guarantee that $N(\bu,P,\bw,\eta)_2 \in \bW'$. Note also that, for $(\bu,\bw) \in \bX\times\bW$ we have
\begin{align*}
  \int_\Omega |\nabla(\bw \times \bu)|^\frac{3}{2}\diff x &\le
    C \left(\int_\Omega |\nabla\bw|^\frac{3}{2} |\bu|^\frac{3}{2}\diff x+\int_\Omega |\nabla\bu|^\frac{3}{2} |\bw|^\frac{3}{2}\diff x \right) \\
    &\leq C \left( \|\nabla\bw\|^\frac{3}{2} \|\bu\|_{\bL^{6}}^\frac{3}{2}+\|\nabla\bu\|^\frac{3}{2} \|\bw\|_{\bL^{6}}^\frac{3}{2} \right)\\
    & \le  C \|\nabla \bu\|^\frac{3}{2} \|\nabla \bw\|^\frac{3}{2}.
\end{align*}
Hence $\bw \times \bu\in \bW^{1,\frac32}(\Omega)$,  which is compactly embedded in $\bL^2(\Omega)$. By the same arguments we  also have that   $(\bw\cdot\nabla)\bu-
(\bu\cdot\nabla)\bw\in\bL^{\frac32}(\Omega)$ for $\bu,\bw\in\bX\times\bW$. Since $\bH^1(\Omega)$ is compactly embedded in $\bL^{3}(\Omega) $, we have the compact embedding  $\bL^{\frac32}(\Omega)=(\bL^{3}(\Omega))'\hookrightarrow\hookrightarrow (\bH^{1}(\Omega))'\subset\bW'$; see \cite[Theorem 5.11.2]{MR3136903}.
Therefore, we  conclude that the operator  $N$ is well defined and compact.

Define now $F: \bX \times Q \times \bW\times\LL \to \bX \times Q \times \bW\times\LL$ as
\begin{equation}\label{F}
F(\bu,P,\bw,\eta)=T(N(\bu,P,\bw,\eta)).
\end{equation}
By a superposition of continuous and compact operators, this defines a compact operator.
We have now established that the operator defined by \eqref{F} is a compact map from $\bX\times Q \times \bW\times\LL$ into itself.  Further, note that solutions of the fixed point problem
\[
(\bu,P,\bw,\eta) = F(\bu,P,\bw,\eta),
\]
are solutions of \eqref{weak1}--\eqref{weak3}.

We are now ready to show the existence of solutions.

\begin{theorem}[existence]
\label{thm:existence_stat}
Assume that $\blf \in \bL^2(\Omega)$ is such that \eqref{eq:f_small} holds, then problem \eqref{weak1}--\eqref{weak3} has at least one solution. Thus, for large enough $\alpha$ the solution exists and is unique.
\end{theorem}
\begin{proof}
Consider the family of fixed point problems, with $\lambda\in [0,1]$:
\[
(\bu_{\lambda},P_{\lambda},\bw_{\lambda},\eta_{\lambda}) = \lambda F(\bu_{\lambda},P_{\lambda},\bw_{\lambda},\eta_{\lambda}).
\]
Decomposing $F$ and noting that $\lambda T(\bg)=T(\lambda \bg)$, we have that
\begin{eqnarray*}
(\bu_{\lambda},P_{\lambda},\bw_{\lambda},\eta_{\lambda}) & = & \lambda T(N(\bu_{\lambda},P_{\lambda},\bw_{\lambda},\eta_{\lambda}))  \\
& = & T( \lambda N(\bu_{\lambda},P_{\lambda},\bw_{\lambda},\eta_{\lambda})),
\end{eqnarray*}
and thus for a given $\lambda \in [0,1]$, solutions to the associated fixed point problem satisfy
\begin{align}
 \alpha (\bu_{\lambda},\bv) + \nu(\nabla \bu_{\lambda},\nabla \bv)+ \lambda (\bw_{\lambda} \times \bu_{\lambda},\bv) - (P_{\lambda},\div \bv)  &=   (\lambda \blf ,\bv), \label{lin1b} \\
(\div \bw_{\lambda},q)=(\div \bu_{\lambda},\pi) & =  0, \label{lin2b} \\
\alpha( \bw_{\lambda}, \bchi) +\nu(\rot \bw_{\lambda},\rot \bchi) + \nu(\div \bw_{\lambda},\div \bchi) + \lambda &b(\bu_{\lambda},\bw_{\lambda},\bchi)  \nonumber \\ - \lambda b(\bw_{\lambda},\bu_{\lambda},\bchi) & =
(\lambda \blf,\rot \bchi)  \label{lin3b}
-f_{\rm bc}(P_{\lambda},\bchi),
\end{align}
for every $(\bv,\pi,\bchi,q)\in  \bX\times Q \times \bW\times\LL$.  To obtain a priori bounds on the solutions, we note that this system is identical to \eqref{weak1}--\eqref{weak3}, except that the right hand side $\blf$ is scaled by $\lambda$, as are the nonlinear terms in the velocity and vorticity equations.  Since $0\le \lambda\le 1$, the same proof for a priori bounds as is done for \eqref{weak1}--\eqref{weak3} in Lemma \ref{lem:aprioribounds} can be repeated for this system, and the only difference in the bounds is the dependence of the constants on $\lambda^k$, with $k\ge 0$.  But since $\lambda \le 1$, the following bounds hold uniformly in $\lambda$:
\[
\| \nabla \bu_{\lambda} \| + \| \nabla P_{\lambda} \| + \| \nabla \bw_{\lambda} \| +\| \eta_{\lambda} \|\le C,
\]
where the constant $C$ depends only on the problem data.
Thus by Lemma \ref{LSlemma}, there exists a fixed point for $F$ and thus a solution for \eqref{weak1}--\eqref{weak3}.
\end{proof}

\section{Conclusions and Outlook}
\label{sec:conclusion}

We have proven well-posedness of a steady {velocity-vorticity} system with no-slip velocity boundary conditions, no penetration vorticity boundary conditions, and a natural boundary condition for vorticity involving a pressure functional. The well-posedness result provides a mathematical foundation for numerical methods, which exploit vorticity equations with these boundary conditions, and suggests a possible framework for their numerical analysis.

Significant technical difficulties arose due to the use of vorticity, vorticity boundary conditions, and the boundary conditions containing a pressure functional, and were overcome by a long, technical analysis that combined and extended analyses from \cite{girault1990curl,GR86,borchers1990equations}.  These results are important for at least two reasons in addition to that stated above: First, there are very few analytical results known for fluid problems involving vorticity boundary conditions. 
Second, several results in section \ref{sec:prelim} are new, and can aid in future works for related PDEs which use non-standard boundary conditions.

Moving forward, these results will allow for improved analysis of numerical methods for NSE in velocity-vorticity formulations, and thus likely also to improved algorithms.  To date, analysis of numerical schemes for these systems has been limited to 2D or the 3D steady case, and without physically derived boundary conditions, e.g. \cite{ARZ17,LOR11,HOR17}.

\section*{Acknowledgments}
The authors wish to thank Prof. Roger Temam for a helpful discussion regarding this work.

\bibliographystyle{plain}
\bibliography{references}

\end{document}